\documentclass[12pt]{article}

\usepackage{graphicx} % Required for inserting images
\usepackage{pgfplots}
\usepackage[cmex10]{amsmath}
\usepackage{amsfonts}
\usepackage{mathrsfs}
\usepackage{mathtools}
\usepackage{amsthm}        % <<< szükséges a \theoremstyle-hez
\usetikzlibrary{arrows}
\usepackage{tikz-cd}
\usepackage{quiver}
\usepackage[a4paper,left=30mm,right=30mm,top=25mm,bottom=25mm]{geometry}

\theoremstyle{plain} % dőlt (default) stílus a tételekhez
\newtheorem{theorem}{Theorem}[section]

\newtheorem{corollary}[theorem]{Corollary}

\newtheorem{problem}[theorem]{Problem}

\theoremstyle{definition} % **ez** biztosítja, hogy a Definition *nem* dőlt legyen
\newtheorem{remark}[theorem]{Remark}

\newtheorem{definition}[theorem]{Definition}

\usepackage{tikz}

\newcommand{\comment}[1]{\typeout{swallowing comment}}

\title{New results about Q and $\Delta$-spaces}
\author{János Balázs Ivanyos and Ákos Székely}
\date{July 2026}

\begin{document}

\maketitle
\begin{abstract}
 A topological space \(X\) is called a \(Q\)-space if every subset of \(X\) is a \(G_\delta\)-set, and \(X\) is a \(\Delta\)-space  if for any decreasing sequence \(\{D_n : n \in\omega\}\) of subsets of \(X\) with empty intersection there is a decreasing sequence \(\{U_n : n \in \omega\}\) of open sets with empty intersection such that \(D_n \subseteq U_n\) for all \(n \in\omega\).

Our main result shows that the following statements are equiconsistent:
\begin{itemize}
    \item[(1)] There exists a measurable cardinal;
    \item[(2)] There exists a crowded Baire \(T_1\) \(\Delta\)-space;
    \item[(3)] There exists a crowded Baire \(T_4\) \(Q\)-space;
    \item[(4)] There exists a \(T_1\) \(\Delta\)-space admitting a strictly positive probability measure vanishing on points;
    \item[(5)] There exists a \(T_3\) \(Q\)-space admitting a strictly positive probability measure vanishing on points.
\end{itemize}
This provides complete answers to some problems and partial answers to other problems that have recently appeared in the literature.

We also prove a new result concerning Lindelöf \(Q\)-spaces: if \(X\) is a \(T_3\) Lindelöf \(Q\)-space with \(w(X)\leq \mathfrak c\), then \(|X|<\operatorname{cf}(\mathfrak c)\). This yields a number of nonexistence results for large Lindelöf, locally compact, compact, and countably compact \(Q\)-spaces.

\end{abstract}
\section{Introduction}
First, recall the definition of $Q$-spaces.
\begin{definition}
A topological space $X$ is called a \emph{$Q$-space} if $X$ is non-$\sigma$-discrete and every subset of is $G_{\delta}$.
\end{definition}

Originally, Q-spaces were considered as  subspaces
of the real line, such subspaces are now usually called  $Q$\emph{-sets}. The existence of a $Q$-set clearly implies $2^{\omega_{1}}=2^{\omega}$, so, for example, under $CH$, no $Q$-set exists. On the other hand, it is known that under Martin's axiom ($MA$) and the negation of $CH$, every uncountable subset  of the real line with cardinality less than $\mathfrak{c}$ is a $Q$-set \cite{Martin}.

Later, Zoltán Balogh \cite{Balogh1} constructed  a $T_3$  $Q$-space in $ZFC$, and by further refining his method, he even gave an example of a $T_4$ paracompact $Q$-space \cite{Balogh2}.

Now recall the definition of $\Delta$-spaces.
\begin{definition}
A topological space \(X\) is a \(\Delta\)-space if for every decreasing sequence of subsets
\(\{D_n : n\in\omega\}\) of \(X\) with empty intersection, there exists a decreasing sequence of open subsets
\(\{U_n : n\in\omega\}\) of \(X\) such that for every \(n \in \omega :\) \(U_n\supseteq D_n\) and
\[
\bigcap_{n\in\omega} U_n = \emptyset.
\]
\end{definition}
It is easy to see that every $Q$-space is also a $\Delta$-space. On the other hand, if $X$ is a $Q$-space, then its \emph{Alexandroff duplicate} is a $\Delta$-space, but not a $Q$-space.

The notion of a $\Delta$\emph{-set} can be defined analogously to that of a $Q$-set: namely, it is an uncountable subset of the real line that is a $\Delta$-space. Furthermore, it follows from \cite[Corollary 4.2]{Soukup} that if $o(X)^{\omega}\leq |X|$ and $\operatorname{cf}(o(X))>\omega$, then $X$ is not a $\Delta$-space. Consequently, every $\Delta$-set has cardinality $<\mathfrak{c}$. Therefore under $MA$, every $\Delta$-set is also a $Q$-set. We note that under $MA$, every subset of the real line of cardinality $<\mathfrak{c}$ is meagre, so the following problem, which appeared in \cite{Kakol} and \cite{isr}, is natural: \emph{Does there exist a crowded Baire} \emph{$\Delta$-space?} In Section \ref{sec: mérhető}, we show that the existence of such
spaces is equiconsistent with the existence of a measurable cardinal.

We also mention that $\Delta$-spaces have recently become a topic of interest in the study of $C_p$-spaces, since it was shown in \cite{distinguished} that a Tychonov space $X$ is a $\Delta$-space if and only if the locally convex space $C_p(X)$ is \emph{distinguished}.

The structure of the paper is as follows: in Section~\ref{sec: mérhető}, as mentioned above, we study crowded Baire $Q$-spaces and $\Delta$-spaces, as well as their measure-theoretic counterparts. In Section~\ref{sec: Lindelöf}, we present a new result concerning Lindelöf $Q$-spaces.

\section{Notations}
Our set-theoretic notation and topological terminology are standard; we follow \cite{Jech} and \cite{Top}, but for the reader's convenience we recall those notions that may not be widely familiar.

\begin{enumerate}
    \item[(1)] Let $X$ be a topological space. Then the minimum cardinality of a nonempty open subset of the space is denoted by $\Delta(X)$.
    \item[(2)] Throughout, by a measure on a topological space \(X\) we mean a measure defined on a \(\sigma\)-algebra \(\mathcal{A}\subseteq\mathcal{P}(X)\). Unless stated otherwise, \(\mathcal{A}\) contains the open sets of \(X\). We say that \(\mu\) is strictly positive if every nonempty open set has positive measure, and that \(\mu\) vanishes on points if \(\mu(\{x\})=0\) for every \(x\in X\).
    \item[(3)] Let $\kappa$ be a cardinal. We denote by $Fn(\kappa, 2)$ the poset of partial functions from $\kappa$ to $2$, equipped with the usual ordering, that is, $p\leq q$ iff $p\supseteq q$. Moreover, $RO(Fn(\omega\times\kappa, 2) )$ denotes the Boolean algebra of regular open sets in the complete Boolean algebra completion of $Fn(\omega\times\kappa, 2)$, which is also commonly
called the Cohen algebra.
    \item[(4)] Let $\kappa$ be a cardinal, and consider the usual product measure $\mu$ on $2^{\omega\times\kappa}$. Let $Meas(2^{\omega\times\kappa})$ be the measure algebra. Then $ \mathbb{B}(\kappa)= Meas(2^{\omega\times\kappa})/ \cal N$ denotes the  random algebra, where $\cal N$ is the ideal of $\mu$-null sets. 
\end{enumerate}

\begin{remark}\label{rem: random}
We note the following regarding (3): if we consider the Haar measure $\mu'$ on $(2^{\omega\times\kappa})$, then, on the one hand, it extends $\mu$, and on the other hand, we obtain that the complete Boolean algebra $\mathbb{B}(\kappa)$ is isomorphic to the algebra $Borel(2^{\omega\times\kappa})/ \cal N'$, where $\cal N'$ is the ideal of $\mu'$-null sets. Indeed, the   Haar measure on $2^{\omega\times\kappa}$ is strongly inner regular (open sets can be approximated from within by compact sets), for every open set $U\subseteq 2^{\omega\times\kappa}$ there exist  basic open sets $B_{n}\in Meas(2^{\omega\times\kappa})$ such that $B_{n}\subseteq U$ for $n\in\omega$, and
$$
U\setminus\big(\bigcup_{n\in\omega} B_{n}\big)\in \cal N'.
$$
Moreover, by the outer regularity of the Haar measure (that is, every Borel set can be approximated from outside by open sets), for every $C\in Borel( 2^{\omega\times\kappa})$ there exist basic open sets $B_{n,k}$ such that
$$
D=\bigcap_{n\in\omega}\big(\bigcup_{k\in\omega} B_{n,k}\big)
$$
and  $C$ and $D$ have symmetric difference in $\cal N'$.
\end{remark}
\section{On measure and category for $Q$ and $\Delta$-spaces}\label{sec: mérhető}
In this section, we answer several problems that have appeared in the literature.  \cite[Problem 6.10]{isr} asks the following: \emph{Is it true that if $X$ admits a strictly positive $\sigma$-additive measure vanishing on points, then $X$ is not a $\Delta$-space?} Recall that a measure is strictly positive if and only if it assigns positive measure to every nonempty open set.

If we interpret the above problem literally, then the answer is negative in general. Indeed, if there exists a $\Delta$-space $X$ such that every nonempty open set is uncountable (that is, $\Delta(X)\geq\omega_1$), then one can define a measure $\mu$ by assigning $\infty$ to uncountable sets and $0$ to countable ones. In this case, $\mu$ is a strictly positive $\sigma$-additive measure vanishing on points on the $\Delta$-space $X$.
Moreover, there exists a $\Delta$-space $X$, for which $\Delta(X)\geq\omega_1$. For example, consider the following space: take $\omega_{1}$ many pairwise disjoint copies of $\mathbb{Q}$, where $\mathbb{Q}$ is the set of rational numbers with the usual topology, and let ${\cal I}$ be the ideal on it consisting of the nowhere dense sets. Endow it with the following topology:
$$
\tau\coloneqq\{ \emptyset\}\cup\{ U\subseteq \bigsqcup_{i\in\omega_{1}} \mathbb{Q} \ : \ \bigsqcup_{i\in\omega_{1}} \mathbb{Q}\setminus U\in {\cal I}\}.
$$
It is easy to see that the resulting space $X$ is a $\Delta$-space with $\Delta(X)=\omega_{1}$.

We note that the above space is only $T_{1}$. However, for example, under Martin's axiom, there is a $Q$-set $X$ with $\Delta(X) \geq \omega_{1}$.

The preceding examples can be regarded as trivial from the measure-theoretic point of view. Thus, we assume that what the authors of \cite{isr} actually intended to ask was the following: \emph{Is it true that if $X$ admits a strictly positive  $\sigma$-additive \emph{probability} measure, vanishing on points, then $X$ is not a $\Delta$-space?} In what follows, we answer this question and also provide a complete solution to   \cite[Problem 4.1]{Kakol} and to the first part of  \cite[Problem 6.9]{isr}: \emph{Does every Baire $\Delta$-space have an isolated point?} We also give a partial answer to the second part of  \cite[Problem 6.9]{isr}:\emph{ Are there  Baire $Q$-spaces?}

\begin{theorem}\label{thm: eqv}
The following are equiconsistent:
\begin{itemize}
    
\item[(1)] There exists a measurable cardinal;
\item[(2)] There exists a crowded Baire $T_{1}$ $\Delta$-space;
\item[(3)] There exists a crowded Baire $T_{4}$ $Q$-space\footnote{We note that a crowded $T_1$ Baire space cannot be $\sigma$-discrete.};
\item[(4)] There exists a $T_{1}$ $\Delta$-space which admits a strictly positive probability measure  vanishing on points;
\item[(5)] There exists a $T_{3}$ $Q$-space which admits strictly positive probability measure vanishing on points.
\end{itemize}
\end{theorem}
\begin{proof}
The implication $Con (2)\implies Con(1)$ was proved in \cite[Corollary~3.2]{Soukup}: if there exists a crowded Baire $T_{1}$ $\Delta$-space, then there is an inner model containing a measurable cardinal. Moreover, $ (3)\implies (2)$ and $(5)\implies (4)$ hold trivially. \medskip

Now we prove the implication $Con (1)\implies Con(3)$. For this, we will modify Shelah's construction in \cite{Borel}, so that the resulting space is a $Q$-space. Assume that in our ground model $V$ there is a measurable cardinal $\kappa$. Let
$$
j: V \longrightarrow M
$$
be the elementary embedding derived from the ultrapower by a measure on $\kappa$, with critical point $\kappa$, and that $M^{\kappa}\subseteq M$ and $|j(\kappa)|=2^{\kappa}$. We add $\kappa$ Cohen reals by forcing with the partial ordering $\mathbb{P}\coloneqq Fn(\omega\times\kappa,2)$. Let $G$ be $\mathbb{P}$-generic over $V$, hence in the forcing extension $V[G]$ we have $2^{\omega}=\kappa$. By elementarity,$$j(\mathbb{P})=Fn(\omega\times j(\kappa),2)\cong Fn(\omega\times\kappa,2)\times Fn(\omega\times(j(\kappa)\setminus \kappa),2)=\mathbb{P}\times \mathbb{S},$$where $\mathbb{S}\coloneqq Fn(\omega\times(j(\kappa)\setminus \kappa),2)$. Let $H$ be $\mathbb{S}$-generic over $V[G]$, then $G\times H$ is $j(\mathbb{P})$-generic over $V$. By \cite[Proposition 9.1]{Cummings}, in $V[G\times H]$ we can lift $j$ to obtain an elementary embedding $j_{G}: V[G]\to M[G\times H]$, with $j_{G}(G)=G\times H$.

From this point on, we work in $V[G]$. We define a $\kappa$-complete ideal ${\cal I}$ on $\kappa$ as follows. 
%Let $\dot{\jmath}$ denote the canonical $\mathbb{S}$-name for the lifted embedding $j_G$. 
For $X\subseteq \kappa$ with $X\in V[G]$, we define:$$X\in{\cal I} \quad\text{iff}\quad 1_{\mathbb{S}}\Vdash^{V[G]} \kappa\notin j_{G}(\check{X}).$$Then the map defined by$$\Phi : [X]_{{\cal I}} \mapsto \| \kappa \in j_{G}(\check{X}) \|_{\mathbb{S}}$$induces an isomorphism between the Boolean algebra $P(\kappa)/{\cal I}$ and $RO(\mathbb{S})$, moreover the ideal $\mathcal I$ is $\kappa$-complete and contains the singletons, see \cite[Lemma 3.2]{Kumar}, or \cite[Chapter 17.1]{Cummings}\footnote{In the latter reference, the statement is proved for $Col(\omega,<\kappa)$, but the exact same analysis applies to the Cohen algebra.}.  

We claim that there exists a family of sets $\{ A_{\alpha} \ : \ \alpha\in2^\kappa\}\subset P(\kappa)$ such that the following hold:
\begin{itemize}
\item[(i)] If $s\in Fn(2^{\kappa}, 2)$ then
$$B_{s}\coloneqq\bigcap_{\alpha\in \mathit{dom}(s)} A^{s(\alpha)}_{\alpha} \notin {\cal I}$$
where $A^{1}_{\alpha}=A_{\alpha}$ and $A^{0}_{\alpha}=\kappa\setminus A_{\alpha}$;

\item[(ii)] For every $C\subseteq\kappa$ there exist $s_{n} \in Fn(2^{\kappa}, 2)$ for $n\in\omega$ such that
$$C/{\cal I} = \bigcup_{n\in\omega} B_{s_{n}}/{\cal I};$$

\item[(iii)] For every $I\in{\cal I}$ there exist $\alpha_{n}\in 2^{\kappa}$ such that
$$I=\bigcap_{n\in\omega} A_{\alpha_{n}};$$

\item[(iv)] For every pair $(I, J)\in {\cal I}^{2}$, if $I\cap J=\emptyset$ then there exists $\alpha\in 2^{\kappa}$ such that $I\subset A_{\alpha}$ and $J\subset A^{0}_{\alpha}$ (or vice versa). \end{itemize}

Conditions (i) and (ii) hold if we choose  an arbitrary representative family $\{ A'_{\alpha} \ : \ \alpha\in 2^{\kappa}\}\subset P(\kappa)$ from the preimage under $\Phi$ of the following generating system of $RO( \mathbb{S})$:$$\{[\langle (n,\nu), 1 \rangle]\in RO( \mathbb{S}) \ : \ n\in\omega, \ \nu\in j(\kappa)\setminus\kappa\}= \{ \Phi([A'_{\alpha}]_{\cal I}) \ : \ \alpha\in 2^{\kappa}\}.$$   Note that if $K\in[2^{\kappa}]^{\omega}$, then
$$
\bigcap_{\alpha\in K} A'_{\alpha}  \in {\cal I},
$$
since the intersection of countably many distinct generators is $0$ in $RO( \mathbb{S})$. 

Since there exists an $I \in \mathcal{I}$ with $|I| = \kappa$, it follows that $|\mathcal{I}| = 2^\kappa$. Re-index the family as
\[
\{A'_{\alpha} : \alpha \in 2^{\kappa}\}
=
\{A'_{I,n} : I \in \mathcal{I},\ n \in \omega\},
\]
with $A'_{I,n} \neq A'_{J,m}$ whenever $\langle I,n\rangle \neq \langle J,m\rangle$. Let
$$
Z_{I} = \bigcap_{n>0} A'_{I, n}\in\cal I
$$
and set $$A_{I, n}=  (A'_{I, n}\setminus Z_{I}) \cup I$$ for $n>0$. Then
$$
 \bigcap_{n\in\omega} A_{I, n} = I.
$$
Furthermore, by modifying the $A'_{I, 0}$'s using sets from $\mathcal{I}$, we can ensure that (iv) holds. Then consider the following re-indexing
\[
\{A_{I,n} : I \in \mathcal{I},\ n \in \omega\}
=
\{A_{\alpha} : \alpha \in 2^{\kappa}\}.
\]
This family has the desired properties.

Consider  the topology on $\kappa$ whose basis is the following set
$$
\{ B_{s} \setminus I  \ : \ s\in Fn ( 2^{\kappa}, 2), \ I\in {\cal I} \}.
$$
Every $I \in \mathcal I$ is closed, since if $x \notin I$, then by (iv) we can find some $\alpha \in 2^\kappa$ such that $x \in A_\alpha$ and $I \subset \kappa\setminus A_\alpha $. Furthermore, since the equivalence classes of the sets $B_s$ form a dense subset in the quotient algebra $\mathcal{P}(\kappa)/\mathcal{I}$, for every $X\subseteq\kappa$, $X\notin\mathcal{I}$ holds precisely when there exist $s\in Fn(2^{\kappa}, 2)$ and $I\in\mathcal{I}$ such that$$B_{s} \setminus I \subseteq X,$$so $\mathcal{I}$ contains all nowhere dense sets. Since ${\cal I}$ is $\omega$-complete,   it contains all the meagre sets. Thus our space is Baire.

By (ii), for every subset $C$ of the space, its symmetric difference with some union $\bigcup_{n\in\omega} B_{s_n}$ belongs to $\mathcal{I}$. By the definition of the basis, this implies that $C$ can be written as the union of an open set and a set from the ideal. Since every $\mathcal{I}$-set is a $G_\delta$-set by (iii), it follows that $C$ is also a $G_\delta$-set. Therefore, our space is a Q-space.

Let $Z$ and $Z'$ be two disjoint closed sets. Then
$$
Z =\big( \bigcup_{n} B_{s_n}\setminus I_{n} \big)\cup J, \quad \text{ and }\quad Z' =\big( \bigcup_{n} B_{s'_n}\setminus I'_{n} \big)\cup J'.
$$
Note that the closure of a $B_{s}\setminus I$ is $B_{s}$, so we may assume that
$$
Z =\big( \bigcup_{n} B_{s_n} \big)\cup J, \quad \text{ and }\quad Z' =\big( \bigcup_{n} B_{s'_n} \big)\cup J'.
$$
Assume
that none of $\bigcup_{n} B_{s_n}$, $\bigcup_{n} B_{s'_n}$, $J$, and $J'$ none is empty (the other cases are similar). Thus there exists $\alpha\in 2^{\kappa}$ such that $J\subset A_{\alpha}$ and $J'\subset A^{0}_{\alpha}$. Then
$$
\big( \bigcup_{n} B_{s_n}\big)\cup\big(  \bigcup_{n} (A_{\alpha}\setminus B_{s'_n})\big)
\quad\text{and}\quad
\big( \bigcup_{n} B_{s'_n}\big)\cup\big(  \bigcup_{n} (A^{0}_{\alpha}\setminus B_{s_n})\big)
$$
are disjoint open sets that separate $Z$ and $Z'$. This  shows that our space is $T_{4}$. This completes the proof of $Con(1)\implies Con(3)$.\medskip

Now we consider the implication $Con (1)\implies Con(5)$. As before, assume that in our ground model $V$ there is a measurable cardinal $\kappa$. Also, let
$$
j: V \longrightarrow M
$$
be the elementary embedding derived from the ultrapower by a measure on $\kappa$, with critical point $\kappa$, and that $M^{\kappa}\subseteq M$ and $|j(\kappa)|=2^{\kappa}$. We add $\kappa$ random reals by forcing with the measure algebra $\mathbb{B}(\kappa)$. Let $G$ be $\mathbb{B}(\kappa)$-generic over $V$. By elementarity,$$j(\mathbb{B}(\kappa))=\mathbb{B}(j(\kappa))\cong \mathbb{B}(\kappa)\times \mathbb{B}(j(\kappa)\setminus \kappa).$$Let $H$ be $\mathbb{B}(j(\kappa)\setminus \kappa)$-generic over $V[G]$, then $G\times H$ is $j(\mathbb{B}(\kappa))$-generic over $V$. Then in $V[G\times H]$ we can lift $j$ to obtain an elementary embedding $j_{G}: V[G]\to M[G\times H]$, with $j_{G}(G)=G\times H$. Now working in $V[G]$, we define a $\kappa$-complete ideal ${\cal I}$ on $\kappa$ as follows.  For $X\subseteq \kappa$ with $X\in V[G]$, we define:$$X\in{\cal I} \quad\text{iff}\quad 1_{\mathbb{B}(j(\kappa)\setminus \kappa)}\Vdash^{V[G]} \kappa\notin j_{G}(\check{X}).$$Then the map defined by$$\Psi : [X]_{{\cal I}} \mapsto \| \kappa \in j_{G}(\check{X}) \|_{\mathbb{B}(j(\kappa)\setminus \kappa)}$$induces an isomorphism between the Boolean algebra $P(\kappa)/{\cal I}$ and $\mathbb{B}(j(\kappa)\setminus \kappa)$, see \cite[Lemma 4.2]{Kumar} and \cite[Chapter 17.1]{Cummings}.

We claim that there exists a family of sets $\{ A_{\alpha} \ : \ \alpha\in2^{\kappa}\}\subset P(\kappa)$ such that the following hold:
\begin{itemize}
\item[(i)] If $s\in Fn(2^{\kappa}, 2)$ then
$$B_{s}\coloneqq\bigcap_{\alpha\in \mathit{dom}(s)} A^{s(\alpha)}_{\alpha} \notin {\cal I}$$ where $A^{1}_{\alpha}=A_{\alpha}$ and $A^{0}_{\alpha}=\kappa\setminus A_{\alpha}$;

\item[(ii)] For every $C\subseteq\kappa$ there exist $s_{n,k} \in Fn(2^{\kappa}, 2)$ such that
$$C/{\cal I} = \bigcap_{n\in\omega}\big(\bigcup_{k\in \omega} B_{s_{n,k}}/{\cal I}\big).$$ 

\item[(iii)] For every $I\in{\cal I}$ there exist $\alpha_{n}\in 2^{\kappa}$ such that
$$I=\bigcap_{n\in\omega} A_{\alpha_{n}};$$

\item[(iv)] For every pair $(I, J)\in {\cal I}^{2}$, if $I\cap J=\emptyset$ then there exists $\alpha\in 2^{\kappa}$ such that $I\subset A_{\alpha}$ and $J\subset A^{0}_{\alpha}$ (or vice versa).
\end{itemize}

To see this, let $\{ A'_{\alpha} \ : \ \alpha\in2^{\kappa}\}\subset P(\kappa)$  be an arbitrary representative family of the preimage under $\Psi$ of the  generating system $$\{[\langle (n,\nu), 1 \rangle]\in \mathbb{B}(j(\kappa)\setminus \kappa) \ : \ n\in\omega, \ \nu\in j(\kappa)\setminus\kappa\} =\{ \Phi([A'_{\alpha}]_{\cal I}) \ : \ \alpha\in 2^{\kappa}\}.$$ of $\mathbb{B}(j(\kappa)\setminus \kappa)$. Then (i) holds. And, as before, we can modify the elements of $\{ A'_{\alpha}  \ : \ \alpha\in 2^{\kappa}\}$ so that they satisfy (iii) and (iv). Condition (ii) is satisfied by Remark~\ref{rem: random}.

As we did above, consider on $\kappa$ the topology whose basis is the following set
$$
\{ B_{s} \setminus I \ : \ s\in Fn ( 2^{\kappa}, 2), \ I\in {\cal I} \}.
$$
Then, by (iii), every set in $\cal I$ is $G_\delta$, and by (ii), every subset of the space is a union of a $G_\delta$ set and a set from the ideal. Therefore, our space is a $Q$-space. By (iv), our space is Hausdorff, and a Hausdorff $Q$-space is also $T_3$.

Let $\mu$ be the product measure on $2^{\omega\times(j(\kappa)\setminus \kappa)}$. Then we define
$$\mu^{*}(X)\coloneqq\mu(\widetilde{ \Psi([X]_{\cal I})})$$
as the pullback measure, where $\widetilde{\Psi([X]_{\mathcal I})}$ is an arbitrary element of the equivalence class of $\Psi([X]_{\mathcal I})$ in $2^{\omega\times (j(\kappa)\setminus \kappa)}$. It is clear that the resulting $\mu^{*}$ is a strictly positive measure vanishing on points defined on all subsets of the space. This completes the proof of $Con(1)\implies Con(5)$.
\medskip

Finally, we prove the implication $Con(4) \implies Con(1)$. It is well known that the existence of a measurable
cardinal and the existence of a real-valued measurable cardinal are equiconsistent \cite{Jech}. 

We claim that if there exists a $T_1$ $\Delta$-space which admits a strictly positive probability measure $\mu$ vanishing on points, then there is a real-valued measurable cardinal. Suppose, toward a contradiction that there does not exist a real-valued measurable cardinal, but there is a $T_1$ $\Delta$-space $X$ which admits a strictly positive probability measure $\mu$ vanishing on points.

We will use the following theorem from \cite[Theorem~2.4]{Mertek}:

\medskip
\noindent\textit{Theorem~2.4 from \cite{Mertek}.} Suppose that $X$ is a set, ${\cal A}$ is a $\sigma$-algebra on it, and ${\cal I}\subseteq {\cal A}$ is a $\sigma$-ideal such that
\begin{itemize}
\item[(i)] ${\cal A}\setminus {\cal I}$ is ccc;
\item[(ii)] if $E\subseteq X$ and $E\notin \cal I$, then $E$ has two disjoint subsets $Z_{1}, Z_{2}$ that cannot be separated by elements of ${\cal A}$.
\end{itemize}
Then there exist infinitely many ppairwise disjoint full outer measure subsets $D_{n}\subseteq X$ (that is, for every $B\in\cal A$, if $D_{n}\subseteq B$ then $X\setminus B \in \cal I$).
\medskip

Let ${\cal A}$ be the $\sigma$-algebra of $\mu$-measurable subsets, and let ${\cal I}$ be the $\sigma$-ideal of null sets. \medskip
\newline
\textbf{Claim.} The triple $( X, {\cal A}, {\cal I})$ satisfies (i) and (ii). \medskip 
\newline
\emph{Proof of the claim.} Since $\mu$ is a probability measure, (i) holds trivially. 
Assume indirectly that (ii) fails, that is, there exists some $E\subseteq X$, $E\notin \cal I$, such that any two disjoint subsets of $E$ can be separated by ${\cal A}$. Since $\mu$ is a probability measure, every subset $Y$ of $X$ has a measurable hull $S(Y)\in\cal A$ (that is, $Y\subseteq S(Y)$ and if $Y\subseteq A$ for some $A\in \cal A$, then $A\setminus S(Y)\in \cal I$). 

For $Y\subseteq E$, let $F(Y)\in\cal A$ be a set that separates $Y$ from its complement in $E$. Then
$$\nu (Y):=\frac{\mu(F(Y)\cap S(E))}{ \mu(S(E))}$$
is a probability measure, defined on all subsets of $E$ and vanishing on points; that is, $|E|$ is a real-valued measurable cardinal. This contradiction proves the claim.\qed\medskip

Therefore, by   \cite[Theorem 2.4]{Mertek}, there exist infinitely many pairwise disjoint full outer measure sets $H_{n}$ in $X$. Let
$$D_{n}=\bigcup_{k\geq n} H_{k}.$$
Then
$$\bigcap_{n\in\omega} D_{n} =\emptyset,$$
and since $X$ is a $\Delta$-space, there is a decreasing sequence of open sets  $U_{n}\supseteq D_{n}$ such that
$$\bigcap_{n\in\omega}U_{n}=\emptyset.$$
Furthermore, it is clear that each $D_{n}$ has outer measure $1$, so $\mu(U_{n})=1$ for every $n\in\omega$, which is impossible by continuity of measure. This completes the proof of the theorem.
\end{proof}

Notice that in the last part of the proof above ($Con(4) \implies Con(1)$), we in fact proved the following measure-theoretic corollary, which is interesting in its own right.

\begin{corollary}

Assume that there is no real-valued measurable cardinal, and let $\mu$ be a $\sigma$-finite (and not identically zero) measure on the space $X$ that vanishes on points. Then there exist infinitely many pairwise disjoint subsets of $X$ of full outer measure.
\end{corollary}

\begin{proof}
Essentially, this was proved in Theorem~\ref{thm: eqv} for $Con(4) \implies Con(1)$. The only properties of $\mu$ we used were that it vanishes on points, and that every $Y \subseteq X$ has a measurable hull. The latter always holds for $\sigma$-finite measures.
\end{proof}

\section{On the cardinality of Lindelöf $Q$-spaces}\label{sec: Lindelöf}
In this section, we study the question of whether a $Q$-space can be Lindelöf. As we have already seen, under $\mathrm{MA+\neg CH} $ there exists a $Q$-set $X\subseteq\mathbb{ R}$, and since the real line is hereditarily Lindelöf, the $Q$-space $X$ is Lindelöf. The following natural question is due to Balogh:
\begin{problem}
Does there exist a  Lindelöf \rm{Q}-space in $ZFC$?
\end{problem}
By a remark made in \cite{Q-L}, if there is a Lindelöf $Q$-space in \rm{ZFC}, then it is necessarily an $L$-space (that is, a nonseparable and hereditarily Lindelöf).  The only $L$-space known in $ZFC$ is the space $\cal L$ constructed by Moore \cite{L}. However, in \cite{Q-L} it is shown that $\cal L$ is not a $Q$-space.

It is not hard to see that if $X$ is a $T_3$ Lindelöf $Q$-space such that $w(X) \leq \mathfrak{c}$, then $|X|\leq \mathfrak{c}$. In the following theorem, we prove a stronger upper bound on the cardinality of such spaces $X$.

\begin{theorem}\label{thm: Lindelöf}
    Let $X$ be a $T_{3}$, Lindelöf $Q$-space such that $w(X) \leq \mathfrak{c}$. Then $|X| < \mathrm{cf}(\mathfrak{c})$.
\end{theorem}
\begin{proof}
First note that if \(Y\) is a \(Q\)-space, then \(L(Y)=hL(Y)\). Indeed, whenever \(Z\in\mathcal F_{\sigma}(Y)\), we have \(L(Z)\leq L(Y)\), and since every subspace of a \(Q\)-space is an \(\mathcal F_{\sigma}\)-set, this applies to all \(Z\subseteq Y\).
 %   By the well known theorem of de Groot, $
%|X|\leq 2^{hL(X)}=2^{\omega}
%$ holds by $X$ being Hausdorff hereditarily Lindelöf.
%We shall see that under the present assumptions this bound can be sharpened.
    
    Let $C(X)$ denote the set of continuous, real-valued functions on X. \medskip
\newline \textbf{Claim.} For the space \(X\), the following inequality holds:
\[
o(X) \leq |C(X)| \leq o(X)^{\omega}.
\]

\vspace{0 em}
\noindent\emph{Proof of the claim.}
First, we show that \(o(X)\leq |C(X)|\). To each \(f\in C(X)\), assign the open set
\(U_f=[f>0]\). This defines a surjection onto the set of nonempty open subsets of
\(X\):  since every subset of $X$ is $\mathcal{F}_{\sigma}$, a given open subset $U \subset X$  can be written as the union of a collection $\{F_{n}: n \in \omega \}$ of closed subsets of $X$. Then, since $X$ is normal, being Lindelöf and regular, Urysohn's lemma implies that for every  $n\in \omega$ we can choose an $f_{n} \in C(X; [0,1])$ such
that
\[
F_n\subseteq [f_n=1]\qquad\text{and}\qquad X\setminus U\subseteq [f_n=0].
\]
Therefore, for the function
\[
f=\sum_{n\in\omega}2^{-n}f_n\in C(X),
\]
we have \(U_f=U\), and hence \(o(X)\leq |C(X)|\).

Next, we show that \(|C(X)|\leq o(X)^{\omega}\). Let \(\langle q_n:n\in\omega\rangle\) be
an enumeration of the rationals. To each \(f\in C(X)\) assign the sequence of open sets
\[
\bigl\langle [f<q_n]:n\in\omega\bigr\rangle.
\]
Clearly, to different functions correspond different sequences, so the mapping is injective. This proves the claim.\qed\medskip

    From this, we have $o(X)^{\omega} = |C(X)|^{\omega}$. Since $w(X) \leq 2^{\omega}$ and $X$ being hereditarily Lindelöf, it follows that $$o(X)^{\omega} \leq \mathfrak{c}.$$ 
   
    It is well known that for perfectly normal spaces, for every $\alpha < \omega_{1}$, $\mathrm{Baire}_{\alpha}(X) = \mathrm{Borel}_{\alpha+1}(X)$, where $\mathrm{Baire}_{\alpha}(X)$ is the class of Baire functions on $X$ of order $\alpha$ and $\mathrm{Borel}_{\alpha}(X)$ is the class of $f: $X$ \to \mathbb{R}$ functions, such that $[f \in U] \in \Sigma_{\alpha}^{0}(X)$, for all $U \subset \mathbb{R}$ open.
    Now, since every subset of $X$ is $\Sigma_{2}^{0}(X)$ we have
    \[
     \mathrm{Baire}_{1}(X) = \mathrm{Borel}_{2}(X) =  \overset {X}\, \mathbb{R}
    \]
    hence $|\mathrm{Baire}_{1}(X)| = \mathfrak{c}^{|X|}$.
    On the other hand, \(\mathrm{Baire}_{1}(X)\) consists of pointwise limits of
continuous functions on \(X\), so
     $|\mathrm{Baire}_{1}(X)| \leq |C(X)|^{\omega}$. Combining the inequalities obtained, we get
    \[
       \mathfrak{c}^{|X|} \leq |C(X)|^{\omega} = o(X)^{\omega} \leq \mathfrak{c}.
    \]
    By Kőnig's lemma, this implies $$|X| < \mathrm{cf}(\mathfrak{c}).$$
    \end{proof}

From the previous theorem, it follows that several types of  $Q$-spaces cannot exist. In the following corollary, we list a few.

\begin{corollary}
Let $X$ be a $T_3$ topological space with $|X|\geq \mathrm{cf}(\mathfrak{c})$. Then:
\begin{itemize}
\item[(i)] if $X$ is Lindelöf and $\chi(X)\leq \mathfrak{c}$, then $X$ is not a $Q$-space;
\item[(ii)] if $X$ is Lindelöf and locally compact, then $X$ is not a $Q$-space;
\item[(iii)] if $X$ is compact, then $X$ is not a $Q$-space;
\item[(iv)] if $X$ is countably compact, then $X$ is not a $Q$-space.
\end{itemize}
\end{corollary}
\begin{proof}
(i) If $X$ is Lindelöf with $\chi(X)\leq \mathfrak{c}$ and also a $Q$-space, then it is hereditarily Lindelöf. Hence $w(X) \leq \mathfrak{c}$ and Theorem \ref{thm: Lindelöf} applies to yield $|X| < \mathrm{cf}(\mathfrak{c})$, a contradiction.

(ii) follows from (i), since hereditary Lindelöf locally compact spaces are first countable.

(iii) follows trivially from (ii).

Finally, (iv) follows from (iii), since, as was shown in \cite[Theorem~2.3]{Soukup}, every $T_3$ countably compact $\Delta$-space is compact, and as previously mentioned, every $Q$-space is also a $\Delta$-space.
\end{proof}

\begin{center}
    { \textbf{Acknowledgements}}
\end{center} The authors are deeply grateful to Lajos Soukup for his valuable consultations and helpful remarks.

\bibliography{references}

\begin{thebibliography}{99}

\bibitem{Balogh1} {\sc Balogh, Z.} {\it There is a Q-set space in ZFC.} Proceedings of the American Mathematical Society (1991): 557--561.

\bibitem{Balogh2} {\sc Balogh, Z.} {\it There is a paracompact Q-set space in ZFC.} Proceedings of the American Mathematical Society 126 (1998): 1827--1833.

\bibitem{Cummings} {\sc Cummings, J.} {\it Iterated forcing and elementary embeddings.} Handbook of Set Theory. Dordrecht: Springer Netherlands, (2009), 775--883.

\bibitem{Top} {\sc Engelking, R.} {\it General Topology.} Heldermann Verlag, Berlin, (1989).

\bibitem{Mertek} {\sc Grzegorek, E., Labuda, I.} {\it On two theorems of Sierpiński.} Archiv der Mathematik 110.6 (2018): 637--644.

\bibitem{Jech} {\sc Jech, T.} {\it Set Theory.} Berlin, Heidelberg: Springer Berlin Heidelberg, 2003.

\bibitem{Soukup} {\sc Juhász, I., van Mill, J., Soukup, L., Szentmiklóssy, Z.} {\it Some new results on $\Delta$-spaces.} arXiv preprint arXiv:2510.04242 (2025).

\bibitem{Kakol} {\sc Kakol, J., Leiderman, A., and Tkachuk, V. V.} {\it Some applications of the $\Delta_1$-property.} Topology and its Applications (2025): 109507.

\bibitem{distinguished}  {\sc Kakol, J., and Leiderman, A.} {\it A characterization of X for which spaces $C_{p}(X)$ are
distinguished and its applications.} Proc. Amer. Math. Soc., series B, 8 (2021): 86--99.

\bibitem{Kumar} {\sc Kumar, A., and Kunen, K.} {\it Induced ideals in Cohen and random extensions.} Topology and its Applications 174 (2014): 81--87.

\bibitem{isr} {\sc Leiderman, A., and Szeptycki, P.} {\it On $\Delta$-spaces.} Israel Journal of Mathematics (2025): 1--30.

\bibitem{Q-L} {\sc Memarpanahi, P., and Szeptycki, P.} {\it $Q$-Sets, $\Delta$-Sets, and $L$-Spaces.} arXiv preprint arXiv:2501.08298 (2025).

\bibitem{Martin} {\sc Miller, A.} {\it Special subsets of the real line.} In The Handbook of Set Theoretic Topology, North Holland, 1984, 201--234.

\bibitem{L} {\sc Moore, J.} {\it A solution to the L space problem.} Journal of the American Mathematical Society 19.3 (2006): 717--736.

\bibitem{Borel} {\sc Shelah, S.} {\it A space with only Borel subsets.} arXiv preprint math/0009047 (2000).

\end{thebibliography}

\bigbreak
{\sc János Balázs Ivanyos,
Eötvös Loránd University,
Institute of Mathematics,
Pázmány Péter stny. 1/C, 1117 Budapest, Hungary}

\emph{Email address: janos.ivanyos@gmail.com}

\bigbreak
{\scÁkos Székely,
Eötvös Loránd University,
Institute of Mathematics,
Pázmány Péter stny. 1/C, 1117 Budapest, Hungary}

\emph{Email address: akos\_szekely@yahoo.com}
\end{document}